\documentclass[11pt]{amsart}
\usepackage[colorlinks=true,linkcolor=blue,anchorcolor=blue,linktocpage=true,urlcolor=blue,citecolor=blue]{hyperref}
\usepackage{amssymb}
\usepackage{mathrsfs}
\usepackage{enumerate}
\usepackage{hyperref}
\usepackage{color}
\newtheorem{theorem}{Theorem}[section]

\newtheorem{proposition}[theorem]{Proposition}
\newtheorem{lemma}[theorem]{Lemma}
\newtheorem{corollary}[theorem]{Corollary}
\newtheorem*{claim}{Claim}

\newtheorem{theoremalph}{Theorem}

 \def\NN{{\mathbb N}} 

 \def\RR{{\mathbb R}}

\def \G{\Gamma}
\def \D{\Delta}

\def \cC{\mathcal{C}}

\def \F{\EuScript{F}}

\def \B{\mathcal{B}}

\def\F {{\mathcal F}}
\def\B {{\mathcal B}}
\def\P{\mathcal P}

\def\G{{\mathcal G}}

\begin{document}

\title[Finiteness of physical measures for diffeomorphisms]{Finiteness of physical measures for diffeomorphisms with multi 1-D centers}

\author[Yongluo Cao]{Yongluo Cao}

\address{Department of Mathematics, Soochow University, Suzhou 215006, Jiangsu, China.}
\address{Center for Dynamical Systems and Differential Equation, Soochow University, Suzhou 215006, Jiangsu, China}

\email{\href{mailto:ylcao@suda.edu.cn}{ylcao@suda.edu.cn}}

\author[Zeya Mi]{Zeya Mi}

\address{School of Mathematics and Statistics,
Nanjing University of Information Science and Technology, Nanjing 210044, Jiangsu, China}
\email{\href{mailto:mizeya@163.com}{mizeya@163.com}}

\thanks{Y. Cao\ was supported by National Key R\&D Program of China(2022YFA1005802) and NSFC 11790274.}
\thanks{Z. Mi\ was supported by National Key R\&D Program of China(2022YFA1007800) and NSFC 12271260.}

\date{\today}


\begin{abstract}
Let $f$ be a $C^2$ diffeomorphism on compact Riemannian manifold $M$ with partially hyperbolic splitting
$$
TM=E^u\oplus E_1^c\oplus\cdots\oplus E_k^c \oplus E^s,
$$
where $E^u$ is uniformly expanding, $E^s$ is uniformly contracting, and ${\rm dim}E_i^c=1,~ 1\le i \le k, ~k\ge 1$.
We prove the finiteness of ergodic physical(SRB) measures of $f$ under the hyperbolicity of Gibbs $u$-states.
\end{abstract}

\maketitle


\section{Introduction}
Let $f$ be a $C^{2}$ diffeomorphism on a compact Riemannian manifold $M$. A \emph{physical measure} of $f$ is an invariant Borel probability $\mu$ on $M$ for which there exists a positive Lebesgue measure set of points $x\in M$ such that for any continuous $\varphi: M\to \RR$ one has
$$
\lim_{n\to +\infty}\frac{1}{n}\sum_{i=0}^{n-1}\varphi(f^i(x))=\int \varphi d\mu.
$$
The set of points $\B(\mu,f)$ for which the above convergence holds is called the basin of $\mu$ w.r.t. $f$. Physical measures were firstly discovered by
Sinai, Ruelle and Bowen in uniformly hyperbolic systems \cite{Sin72, Rue76, Bow75, BoR75}. 
In smooth ergodic theory, the physical measures are usually produced from \emph{Sinai-Ruelle-Bowen (SRB)} measures.
One says that an invariant measure is an SRB measure if it admits positive Lyapunov exponents almost everywhere, and its conditional measures along Pesin unstable manifolds are absolutely continuous w.r.t Lebesgue measures on these manifolds.
It is well known that every ergodic SRB measure with non-zero Lyapunov exponents is a physical measure. It remains to be a changeable problem in finding physical measures for systems beyond hyperbolicity ( see \cite[Chapter 11]{BDV05} and \cite{pa,v2,y02} for related conjectures ).
 
In this paper, we are interested in studying physical (SRB) measures for partially hyperbolic systems (see $\S$ \ref{dpe} for more details) under the hyperbolicity of \emph{Gibbs $u$-states}. By Gibbs $u$-state, we mean an invariant measure whose disintegration along strong unstable manifolds are absolutely continuous w.r.t. Lebesgue measures on these manifolds. 
The main result of this paper is the following: 

\begin{theoremalph}\label{TheoA}
Let $f$ be a $C^2$ diffeomorphism with a partially hyperbolic splitting 
$
TM=E^u\oplus_{\succ} E_1^c\oplus_{\succ}\cdots\oplus_{\succ} E_k^c \oplus_{\succ} E^s,
$
where ${\rm dim}E_i^c=1,~ 1\le i \le k, ~k\ge 1$.
If every Gibbs $u$-state of $f$ is hyperbolic\footnote{An invariant measure is hyperbolic if all its Lyapunov exponents are non-zero.},
then there are finitely many ergodic physical(SRB) measures of $f$.
\end{theoremalph}

According to \cite{CSY15}, partially hyperbolic splitting as in Theorem \ref{TheoA} is abundant among diffeomorphisms away from homoclinic tangencies. Partially hyperbolic diffeomorphisms in Theorem \ref{TheoA} sometimes inherit some properties in uniform hyperbolicity (see e.g. \cite{D12,YZ22}). 

We would like to list some related results below:
\begin{itemize}
\item In \cite[Theorem C]{CMY22}, the authors prove the existence of (ergodic) SRB measures for partially hyperbolic diffeomorphisms with multi 1-D centers.
\item Under the hyperbolicity of Gibbs $u$-state, for case $k=1$, it is shown in \cite{HYY19} that there exist finitely many physical measures whose basins cover a full Lebesgue measure subset. 
\item Without the restriction on dimension of center direction, the existence and finiteness of physical/SRB measures have been achieved for partially hyperbolic diffeomorphisms for which the central Lyapunov exponents of Gibbs $u$-states are all negative, or all positive; including  diffeomorphisms with mostly contracting center \cite{bv}; with mostly expanding center \cite{AV15}; with mostly expanding center and mostly expanding center simultaneously \cite{MCY17}.
\end{itemize}

Observe that every SRB measure must be a Gibbs $u$-state, using the result of \cite{CMY22}, the hyperbolicity on Gibbs $u$-states guarantees the existence of ergodic physical measures for diffeomorphisms in Theorem \ref{TheoA}. Hence, the main effort on Theorem \ref{TheoA} is to show the finiteness of ergodic physical measures. 
Let us remark that in the previous work \cite{bv,AV15,MCY17}, the signs of central Lyapunov exponents w.r.t. Gibbs $u$-states are definite. Our assumption on Gibbs $u$-states is more flexible: different Gibbs $u$-states may admit different signs of Lyapunov exponents on the same central sub-bundle. 

In \cite{HYY19}, besides the finiteness of ergodic physical measures, the authors also show the basin covering property of these physical measures, i.e., Lebesgue almost every point belongs to the basin of some ergodic physical measure. 
It is natural to expect that the basin covering property can hold also for the general case $k>1$.  
The main difference between the special case $k=1$ and the general case $k>1$ is that when $k=1$, the non-uniform behavior occurs only in one fixed center. 

Our strategy of the proof of Theorem \ref{TheoA} is to decompose the ergodic physical/SRB measures into different levels by their signs of Lyapunov exponents on centers (see Theorem \ref{eb}), and then show the finiteness of these measures on each level. To achieve this, we will make an effort on getting the closed property of Gibbs $cu$-states, which is done by showing the uniformity on central Lyapunov exponents for Gibbs $cu$-states on the same level(see Theorem \ref{C}).

\section{Preliminary}

\subsection{Dominated splitting and partially hyperbolicity}\label{dpe}
Let $f$ be a $C^1$ diffeomorphism on compact Riemannian manifold $M$. Given two $Df$-invariant subbundles $E$ and $F$, we say $E$ \emph{dominates} $F$ if 
for every $x\in M$,  $u\in E(x)$ and $v\in F(x)$ with $\|u\|=\|v\|=1$, one has that
$$
\|Df(x)v\|< \frac{1}{2} \|Df(x)u\|.
$$
We write $E\oplus_{\succ}F$ to denote this domination. A $Df$-invariant splitting $TM=E\oplus_{\succ} F$ of the tangent bundle is called a \emph{dominated splitting} if $E$ dominates $F$. 

One says that $f$ is a \emph{partially hyperbolic} diffeomorphism if there exists a dominated splitting 
$TM=E^u\oplus_{\succ} E_1\oplus_{\succ}\cdots \oplus_{\succ} E_k\oplus_{\succ} E^s$ 
on the tangent bundle such that
\begin{itemize}
\item $E^u$ is uniformly expanding and $E^s$ is uniformly contracting.
\item at least one of $E^u$ and $E^s$ is nontrivial.
\end{itemize}

\subsection{Gibbs $u$-states and Gibbs $E$-states}
Let $f$ be a $C^2$ diffeomorphism with partially hyperbolic splitting $TM=E^u\oplus_{\succ} E^{cs}$. 
Recall that an $f$-invariant measure is a Gibbs $u$-state if its disintegration along strong unstable manifolds is absolutely continuous w.r.t. Lebesgue measures on these manifolds. Gibbs $u$-states were introduced and established by Sinai-Pesin \cite{ps82} for partially hyperbolic systems. 

We list following properties of Gibbs $u$-states, see \cite[Subsection 11.2]{BDV05} for details.

\begin{proposition}\label{uc}
Let $f$ be a diffeomorphism with partially hyperbolic splitting $TM=E^u\oplus_{\succ} F$, one has following properties:
\begin{enumerate}
\item\label{11} The ergodic components of any Gibbs $u$-state are Gibbs $u$-states.
\item\label{22} The set of Gibbs $u$-states is compact in weak$^*$-topology.
\item\label{33} For Lebesgue almost every $x\in M$, any limit measure of 
$$
\big\{\frac{1}{n}\sum_{i=0}^{n-1}\delta_{f^i(x)}\big\}_{n\in \NN}
$$ 
is a Gibbs $u$-state.
\end{enumerate}
\end{proposition}

Recall the following entropy formula for diffeomorphisms with dominated splittings, see \cite[Theorem F]{CYZ18}, \cite[Theorem A]{cce15} for the proof.\begin{proposition}\label{cug}
Let $f$ be a diffeomorphism with dominated splitting $TM=E\oplus_{\succ} F$. Then for Lebesgue almost every $x\in M$, any limit measure $\mu$ of $\{\frac{1}{n}\sum_{i=0}^{n-1}\delta_{f^i(x)}\}_{n\in \NN}$ satisfies 
$$
h_{\mu}(f)\ge \int \log |{\rm det}Df|_{E}| {\rm d} \mu.
$$
\end{proposition}

Recall another notion called \emph{Gibbs $cu$-states}.
Let $f$ be a diffeomorphism with dominated splitting $TM=E\oplus_{\succ} F$. An invariant measure $\mu$ called a Gibbs $cu$-state associated to $E$ if
\begin{itemize}
\item the Lyapunov exponents of $\mu$ along $E$ are positive, 
\item the conditional measure of $\mu$ along Pesin unstable manifolds tangent to $E$ are absolutely continuous w.r.t. Lebesgue measure on these manifolds.
\end{itemize}

Since we need to deal with various Gibbs $cu$-states for partially hyperbolic diffeomorphisms with several sub-bundles, for simplicity we will just call $\mu$ a Gibbs $E$-state if it is a Gibbs $cu$-state associated to $E$.
Note that when $E$ is uniformly expanding, then Gibbs $E$-state is indeed the Gibbs $u$-state. 

Similar to property (\ref{11}) of Proposition \ref{uc}, we have the following result. See a proof in \cite[Lemma 2.4]{cv} or \cite[Proposition 4.7]{CMY22}.
\begin{proposition}\label{cue}
Assume that $f$ is a $C^2$ diffeomorphism with a dominated splitting $TM = E\oplus_{\succ} F$. If $\mu$ is a Gibbs $E$-state, then any ergodic component of $\mu$ is also a Gibbs $E$-state.
\end{proposition}

We will use the following fact several times, which can be deduced by the absolute continuity of Pesin stable lamination, see \cite{y02} for a proof.
\begin{proposition}\label{pes}
Assume that $f$ is a $C^2$ diffeomorphism with a dominated splitting $TM = E\oplus_{\succ} F$. If $\mu$ is an ergodic Gibbs $E$-state whose Lyapunov exponents along $F$ are all negative, then $\mu$ is a physical measure. 
\end{proposition}


The result below is abstract from \cite[Theorem 4.10]{CMY22}.
\begin{proposition}\label{ugu}
Let $f$ be a diffeomorphism with dominated splitting $TM=E\oplus_{\succ} F$. If $\mu$ is a limit measure of the set of ergodic Gibbs $E$-states, whose Lyapunov exponents along $E$ are uniformly positive, then $\mu$ is a Gibbs $E$-state.
\end{proposition}

\section{Closed property on Gibbs $cu$-states}

For the rest, we will concentrate on partially hyperbolic diffeomorphisms as in Theorem \ref{TheoA}. For convenience, given $k\ge 1$, denote by $\P_k(M)$ the collection of $C^2$ diffeomorphisms in $M$ with partially hyperbolic splitting 
$
TM=E^u\oplus_{\succ} E_1^c\oplus_{\succ}\cdots\oplus_{\succ} E_k^c \oplus_{\succ} E^s
$
such that 
\begin{itemize}
\item ${\rm dim}E_i=1, 1\le i \le k$;
\smallskip
\item every Gibbs $u$-state of $f$ is hyperbolic.
\end{itemize}

We write
$$
F_i=E^u\oplus_{\succ} \cdots \oplus_{\succ} E_i^c, \quad 0\le i \le k.
$$
For every $0\le i \le k$, let $G_{i}$ be the set of Gibbs $F_i$-states of $f$; let $G_i^{erg}$ be the set of ergodic Gibbs $F_i$-states of $f$.
For any $f$-invariant measure $\mu$, by Oseledec's theorem \cite{bp02}, for every $1\le i \le k$, one can define
$$
\lambda^c_i(x)=\lim_{n\to +\infty}\frac{1}{n}\log \|Df^n|_{E^c_i(x)}\|, \quad \mu \textrm{-a.e.}~x.
$$
Moreover, define
$$
\lambda^c_i(\mu)=\int \log\|Df|_{E^c_i}\|{\rm d} \mu.
$$

%
The main result of this section is as follows:
\begin{theorem}\label{34}
If $f\in \P_k(M)$, then for every $0\le i \le k-1$, if $\mu$ is a limit measure of $G_i^{erg}$, then
\begin{itemize}
\item $\mu$ is a Gibbs $F_i$-state,
\item and one has that
$$
\lambda^c_{i+1}(x)> 0, \quad \mu \textrm{-a.e.}~x.
$$
\end{itemize} 
Moreover, there are at most finitely many ergodic Gibbs $F_k$-states.
\end{theorem}

We start with the following observation on Gibbs $E$-states for general diffeomorphisms with dominated splitting $TM=E\oplus_{\succ} F$.

\begin{proposition}\label{sec}
Let $f$ be a diffeomorphism with dominated splitting $TM=E\oplus_{\succ} F$. Then, we have 
\begin{itemize}
\item either there exist at most finitely many ergodic Gibbs $E$-states,
\item or, if $\mu$ is a limit measure of the set of ergodic Gibbs $E$-states, which is a Gibbs $E$-state, then for $\mu$-a.e. $x\in M$, there exist non-negative Lyapunov exponents along $F$.
\end{itemize}
\end{proposition}

\begin{proof}
Suppose that there are infinitely many ergodic Gibbs $E$-states.
Let $\{\mu_n\}_{n\in \NN}$ be a sequence of different ergodic Gibbs $E$-states, which converges to $\mu$ as $n\to +\infty$, and $\mu$ is a Gibbs $E$-state. 

Arguing by contradiction, assume $A\subset M$ satisfying $\mu(A)>0$ and
$$
\lim_{n\to +\infty}\frac{1}{n}\log \|Df^n|_{F(x)}\|<0,\quad \forall~ x\in A.
$$ 
This implies that there exist ergodic components of $\mu$ whose Lyapunov exponents along $F$ are all negative. Note that these ergodic measures are Gibbs $E$-states by applying Proposition \ref{cue}. By Proposition \ref{pes}, we know that they are also physical measures, so there are at most countably many such measures. Consequently, one can find an ergodic component $\nu$ of $\mu$ such that all its Lyapunov exponents along $F$ are negative, and there exists $b\in (0,1]$ such that $\mu$ can be rewritten as
$$
\mu=b \nu+(1-b)\eta
$$
for some Gibbs $E$-state $\eta$. 

\medskip
By Pesin theory \cite[$\S2$ \& $\S4$ ]{bp02}, consider $\{X_{k}\}_{k\ge 1}$ as the Pesin blocks related to $\nu$ with following properties:
\begin{itemize}
\item $X_{k}$ is compact for every $k\ge 1$, and $\nu(X_{k})\to 1$ as $k\to +\infty$;
\smallskip
\item there is $\delta_k>0$ such that every $x\in X_{k}$ admits a local Pesin stable manifold $\F^s_{\delta_k}(x)$ and a local Pesin unstable manifold $\F^u_{\delta_k}(x)$ with size $\delta_k$;
\smallskip
\item $\F^u_{\delta_k}(x)$ is tangent\footnote{This means that $T_y\F^u_{\delta_k}(x)=E(y)$ for every $y\in \F^u_{\delta_k}(x)$.} to $E$ and $\F^s_{\delta_k}(x)$ is tangent to $F$;
\smallskip
\item both $\F^u_{\delta_k}(x)$ and $\F^s_{\delta_k}(x)$ depend continuously on $x\in X_{k}$.
\end{itemize}
Choosing $k\in \NN$ such that
$
\nu(X_{k})>0.
$
Fixing $r\ll \delta_k$. By compactness of $X_{k}$, one can take $p\in X_{k}$ such that 
\begin{equation}\label{ps}
\nu(X_{k}\cap B(p,r))>0.
\end{equation}
Let
$$
S^u=\bigcup_{x\in X_{k}\cap B(p,r)}\F_{\delta_k}^u(x).
$$
We have $\nu(S^u)>0$ by (\ref{ps}). According to the properties of $\{\F_{\delta_k}^u(x)\}_{x\in X_k}$, one knows that 
$$
\mathcal{S}^u=\{\F_{\delta_k}^u(x): x\in X_{k}\cap B(p,r)\}
$$
is a measurable partition of the set $S^u$ w.r.t. $\nu$.

\begin{claim}
${\rm Leb}_{\gamma}(B(\nu,f))>0$ for every $\gamma\in \mathcal{S}^u.$
\end{claim}
\begin{proof}[Proof of the Claim]
Note that the basin $\B(\nu,f)$ has full $\nu$-measure as $\nu$ is ergodic, and $X_{k}\cap B(p,r)\subset S^u$ by construction, we get
$$
\nu(B(p,r)\cap X_{k}\cap \B(\nu,f)\cap S^u)>0
$$
Since $\nu$ is an ergodic Gibbs $E$-state, the conditional measures of $\nu|_{S^u}$ along unstable disks of $\mathcal{S}^u$ are absolutely continuous w.r.t. Lebesgue measures on these disks, which ensures that one can take $\gamma_{\nu}\in \mathcal{S}^u$ for which there exists a subset $B$ of $B(p,r)\cap\gamma_{\nu}\cap X_{k}\cap \B(\nu,f)$ such that 
\begin{equation}\label{n}
{\rm Leb}_{\gamma_{\nu}}(B)>0.
\end{equation}
Since $r$ is sufficiently small, we conclude that for every $x\in X_{k}\cap B(p,r)$, $\F^s_{\delta_k}(x)$ intersects every local Pesin unstable manifold from $\mathcal{S}^u$ transversely. By applying the absolute continuity of Pesin stable lamination, (\ref{n}) ensures
$$
{\rm Leb}_{\gamma}\left(\bigcup_{z\in B}\F_{\delta_k}^s(z)\cap \gamma\right)>0, \quad \textrm{for every}~\gamma \in  \mathcal{S}^u.
$$
Together with the observation that $\F_{\delta_k}^s(z)\subset \B(\nu,f)$ for every $z\in B$, this yields the desired result. 
\end{proof}
Recall that we have $\nu(S^u)>0$, so $\mu(S^u)\ge b\nu(S^u)>0$. Up to reducing the sizes of unstable disks of $\mathcal{S}^u$ slightly, we may assume 
$
\mu(\partial(S^u))=0,
$
which implies that 
$$
\lim_{n\to +\infty}\mu_n(S^u)=\mu(S^u).
$$
In particular, there exists $n_0\in \NN$ such that $\mu_n(S^u)>0$ for every $n\ge n_0$.
For every $n\ge n_0$, since $\mu_n$ is an ergodic Gibbs $E$-state, one can find $\gamma_n\in \mathcal{S}^u$ such that ${\rm Leb}_{\gamma_n}$-almost every point of $\gamma_n$ belongs to $\B(\mu_n,f)$.
This together with the Claim above yields that $\mu_n=\nu$ for every $n\ge n_0$, which is contrary to our assumption.


\end{proof}

We recall the following folklore result \cite[Proposition 4.9]{CMY22}.
\begin{lemma}\label{sim}
Under the setting as above, we have $G_k\subset G_{k-1}\subset \cdots \subset G_0$.
\end{lemma}

Given an invariant measure $\mu$, denote by $\mathcal{C}(\mu)$ the family of ergodic components of $\mu$.
With above preparations, we can give the next result, which asserts a closed property on Gibbs $F_i$-states.


\begin{theorem}\label{C}
If $f\in \P_k(M)$, then for every $1\le i \le k$, either $\# G_i^{erg}<+\infty$ or any limit measure of $G_i^{erg}$ is contained in $G_i$, whose Lyapunov exponents along $E_i^c$ are uniformly positive.
\end{theorem}

\begin{proof}
We prove the result by induction on $1\le i \le k$. 
We start by proving the case $i=1$. Assume that $\# G_1^{erg}=+\infty$ and let $\mu$ be a limit measure of $G_1^{erg}$.  
That is,
there is a sequence of ergodic Gibbs $F_1$-states $\{\mu_n\}_{n\in \NN}$ such that 
$$
\mu_n\xrightarrow{weak^*}\mu, \quad \textrm{as}~n\to +\infty.
$$
We have $\{\mu_n\}_{n\in \NN}\subset G_0$ as $G_1\subset G_0$ by Lemma \ref{sim}. 
From item (\ref{22}) of Proposition \ref{uc}, we know $\mu\in G_0$.
By Proposition \ref{sec} and the hyperbolicity assumption on Gibbs $u$-states, we get
$$
\lambda^c_1(x)>0,\quad \mu \textrm{-a.e.}~x.
$$

To complete the proof of the case $i=1$, by Proposition \ref{ugu} it suffices to verify that the Lyapunov exponents of $\mu$ are uniformly positive along $F_i$. By contradiction, there exists a sequence of ergodic measures $\{\nu_n\}_{n\in \NN}$ in  $\cC(\mu)$ such that 
$$
\lim_{n\to +\infty}\lambda^c_1(\nu_n)=0.
$$
%
Up to considering subsequences, we assume that $\nu_n$ converges to an $f$-invariant measure $\nu$ as $n\to +\infty$. Consequently, we have
\begin{equation}\label{gud}
\lambda^c_1(\nu)=\lim_{n\to +\infty} \lambda^c_1(\nu_n)=0.
\end{equation}
Since we have shown $\mu\in G_0$, and $\{\nu_n\}_{n\in \NN}$ are ergodic components of $\mu$, the first item (\ref{11}) of Proposition \ref{uc} gives that $\nu_n\in G_0$ for every $n\in \NN$. So, $\nu$ is a limit measure of the set of ergodic Gibbs $u$-states. Using Proposition \ref{sec} again and the hyperbolicity of Gibbs $u$-states, we deduce that $\lambda^c_1(x)>0$ for $\nu \textrm{-a.e.}~ x$. It follows from 
Birkhoff's ergodic theorem that 
$$
\lambda^c_1(\nu)=\int \lambda^c_1(x) {\rm d}\nu>0,
$$
where we use the fact ${\rm dim}E_1^c=1$.
This gives the contradiction to (\ref{gud}). Thus, the case $i=1$ is verified.

\medskip
Inductively, we assume that the theorem holds for $i=\ell$($\ell<k$), and we then show that it is true for $i=\ell+1$. Assume that $\# G_{\ell+1}^{erg}=+\infty$. Let $\{\mu_n\}_{n\in \NN}$ be a sequence of ergodic measures in $G_{\ell+1}$, which converges to $\mu$ as $n\to +\infty$. In view of Proposition \ref{ugu}, we need only to show
\begin{equation}\label{uni2}
\inf_{\nu\in \cC(\mu)}\left\{\lambda^c_{\ell+1}(\nu)\right\}>0.
\end{equation}

As Lemma \ref{sim} tells us that $G_{\ell+1}\subset G_{\ell}\subset G_0$ and we assume that the theorem is true for $i=\ell$, we get $\mu\in G_{\ell}$. 
By Theorem \ref{sec}, together with the assumption of Gibbs $u$-states we know that $\mu$ admits only positive Lyapunov exponents along $E_{\ell+1}^c$.

Now we show (\ref{uni2}) is true. By contradiction we assume that there exists a sequence of ergodic measures $\{\nu_n\}_{n\in \NN}\subset \cC(\mu)$, which converges to an $f$-invariant measure $\nu$ and satisfies
\begin{equation}\label{et}
\lim_{n\to +\infty}\lambda^c_{\ell+1}(\nu_n)=\lambda^c_{\ell+1}(\nu)=0.
\end{equation}
Since $\mu\in G_{\ell}$, by Proposition \ref{cue} we know that $\{\nu_n\}_{n\in\NN}$ is a sequence of ergodic measures in $G_{\ell}\subset G_0$. 

As we have assumed that this theorem is true for $i=\ell$, the limit measure $\nu$ of $\{\nu_n\}_{n\in \NN}$ is also a Gibbs $F_{\ell}$-state. By Proposition \ref{sec}, we have
$$
\lambda^c_{\ell+1}(x)\ge 0, \quad \nu \textrm{-a.e.}~x.
$$
Noting that $\nu\in G_0$ guaranteed by item (\ref{22}) of Proposition \ref{uc}, the assumption on Gibbs $u$-states then gives
$$
\lambda^c_{\ell+1}(x)> 0, \quad \nu \textrm{-a.e.}~x.
$$
Therefore, using the Birkhoff's ergodic theorem again, we get
$$
\lambda^c_{\ell+1}(\nu)=\int \lambda^c_{\ell+1}(x) {\rm d}\nu>0,
$$
contradicting to the convergence (\ref{et}). 
Therefore, (\ref{uni2}) is true and we complete the proof of Theorem \ref{C}.

\end{proof}

We now ready to prove Theorem \ref{34}.


\begin{proof}[Proof of Theorem \ref{34}]
For $0\le i \le k-1$, let $\mu$ be a limit measure of $G_i^{erg}$.
We know from Proposition \ref{uc} and Theorem \ref{C} that $\mu$ is a Gibbs $F_i$-state, which is also a Gibbs $u$-state by Lemma \ref{sim}. So, $\mu$ is hyperbolic by definition of $\P_k(M)$. This together with Proposition \ref{sec} ensures that for $\mu$-a.e. $x\in M$, there exist positive Lyapunov exponents along $F:=E_{i+1}^c\oplus_{\succ} \cdots \oplus_{\succ} E^s$. As $E_{i+1}^c$ has dimension one, $\lambda^c_{i+1}(x)$ is just the largest Lyapunov exponent along $F$.
Consequently, we obtain 
$$
\lambda^c_{i+1}(x)> 0, \quad \mu \textrm{-a.e.}~x.
$$ 

To show the result on Gibbs $F_k$-states, assume by contradiction that there are infinitely many ergodic Gibbs $F_k$-states. This implies that there exists a sequence of different ergodic Gibbs $F_k$-states, which converges to an invariant measure $\mu$. By Theorem \ref{C}, $\mu$ must be a Gibbs $F_k$-state as well. However, from Proposition \ref{sec} we get that $\mu$ have positive Lyapunov exponents along $E^s$, which is a contradiction. Thus, the proof is complete.
\end{proof}

\medskip
As a byproduct of Theorem \ref{C}, we also have the following interesting result, which asserts the uniformity on Lyapunov exponents of Gibbs $F_i$-states along $F_i$, $1\le i \le k$. 
\begin{corollary}
Let $f$ be a $C^2$ diffeomorphism as in Theorem \ref{C}, then there exists $\alpha>0$ such that for every Gibbs $F_i$-state $\mu$,
one has that 
$$
\lambda^c_i(x)>\alpha, \quad \mu \textrm{-a.e.}~x.
$$
\end{corollary}

\begin{proof}
Assuming contrary, there exists $1\le i \le k$ and a sequence of Gibbs $F_i$-states $\{\mu_n\}_{n\in \NN}$ for which there exists an ergodic component $\nu_n$ of $\mu_n$ for every $n\in \NN$ so that
$$
\lim_{n\to +\infty}\lambda^c_i(\nu_n)=0.
$$
%
Without loss of generality, we suppose that $\{\nu_n\}_{n\in \NN}$ converges to some $f$-invariant measure $\nu$ when $n\to +\infty$. Note that $x\mapsto \log \|Df|_{E_i^c(x)}\|$ is continuous. By the weak$^*$ convergence we have
\begin{equation}\label{eu}
\lambda^c_i(\nu)=\lim_{n\to +\infty}\lambda^c_i(\nu_n)=0.
\end{equation}
By Proposition \ref{cue}, we know each $\nu_n$ is an ergodic Gibbs $F_i$-state for every $n\in \NN$. Hence, $\nu$ is a limit measure of the set of ergodic Gibbs $F_i$-states. By Theorem \ref{C}, there exists $a>0$ such that 
$$
\lambda^c_i(x)>a,\quad \nu \textrm{-a.e.}~x.
$$
Noting ${\rm dim}E_i^c=1$, by applying Birkhoff's ergodic theorem (see e.g. \cite[Theorem 1.14]{wa82}) one obtains
$$
\lambda_i^c(\nu)=\int \lambda_i^c(x){\rm d}\nu>a,
$$
which is in contradiction with convergence (\ref{eu}).
\end{proof}

\section{Proof of Theorem \ref{TheoA}}

Now we begin to prove Theorem \ref{TheoA}. Assume that $f\in \P_k(M)$.
%
For every $0\le i \le k$, put 
$$
\G_i=\left\{\mu: \mu\in G_i ~\textrm{is ergodic and}~\lambda^c_{i+1}(\mu)<0\right\}.
$$
%
Recall that $G_i$ is the space of Gibbs $F_i$-states for every $0\le i \le k$, respectively; and 
$$
\lambda^c_{i+1}(\mu)=\int \log \|Df|_{E_{i+1}^c}\| {\rm d} \mu.
$$
Note that here we set $E_{k+1}^c=E^s$, and
$\lambda^c_{k+1}(\mu)$ denotes the maximal Lyapunov exponent of $\mu$ along $E^s$, which is negative automatically as the uniform contraction on $E^s$.

\medskip
To give the proof of Theorem \ref{TheoA}, we need the following result.

\begin{theorem}\label{eb}
Under the setting of Theorem \ref{TheoA}, the set of ergodic physical(SRB) measures coincides with $\bigcup_{0\le i \le k}\G_i$.
\end{theorem}

\begin{proof}
To begin with, we show the existence of ergodic physical measures.
From the result \cite[Theorem C]{CMY22}, there exists an ergodic SRB measure $\mu$ for $f$, which is a Gibbs $u$-state by definition. We know also that $\mu$ is hyperbolic by assumption. Consequently, $\mu$ is an ergodic physical measure by using the absolute continuity of Pesin stable lamination. 

By definition, for any $0\le i \le k$, any invariant measure in $\G_i$ is a hyperbolic ergodic SRB measure, which is also a physical measure by Proposition \ref{pes}. Thus, it suffices to show that every ergodic physical measure is contained in $\G_i$ for some $0\le i \le k$.

By item (\ref{33}) of Proposition \ref{uc} and Proposition \ref{cug}, there is a subset $\D$ of $M$ with full Lebesgue measure such that for any $x\in \D$, any limit measure $\eta$ of $\{\frac{1}{n}\sum_{i=0}^{n-1}\delta_{f^i(x)}\}_{n\in \NN}$ is a Gibbs $u$-state and satisfies
\begin{equation}\label{ee}
h_{\eta}(f)\ge \int \log |{\rm det}Df|_{F_{i}}| {\rm d} \eta,\quad \forall~ 1\le i \le k.
\end{equation}

Assume that $\mu$ is an ergodic physical measure, thus $\B(\mu,f)$ admits positive Lebesgue measure. Thus $\B(\mu,f)\cap \D$ exhibits positive Lebesgue measure. By taking $x\in \B(\mu,f)\cap \D$, one gets form definition of $\B(\mu,f)$ that
$$
\lim_{n\to +\infty}\frac{1}{n}\sum_{j=0}^{n-1}\delta_{f^j(x)}=\mu.
$$
So, $\mu$ is a Gibbs $u$-state and satisfies (\ref{ee}) form the definition of $\D$.

Since any Gibbs $u$-state is hyperbolic by assumption, $\mu$ is an ergodic hyperbolic Gibbs $u$-state. Consequently, one can take $i\in \{0,\cdots,k\}$ such that all the Lyapunov exponents of $\mu$ along $F_i$ are positive and $\lambda^c_{i+1}(\mu)<0$. By Ruelle's inequality, one has
$$
h_{\mu}(f)\le \int \log |{\rm det}Df|_{F_{i}}| {\rm d} \mu.
$$
On the other hand, inequality (\ref{ee}) says that
$$
h_{\mu}(f)\ge \int \log |{\rm det}Df|_{F_{i}}| {\rm d} \mu.
$$
Therefore, 
$$
h_{\mu}(f)= \int \log |{\rm det}Df|_{F_{i}}| {\rm d}\mu.
$$
By the result of Ledrappier-Young \cite[Theorem A]{ly}, $\mu$ is an SRB measure, which is contained in $\G_i$ by its sign of Lyapunov exponents.
\end{proof}

Now, we can give the proof of Theorem \ref{TheoA}.

\begin{proof}[Proof of Theorem \ref{TheoA}]

In view of Theorem \ref{eb}, it suffices to show that for each $0\le i \le k$, $\G_i$ is a finite set. We know from Theorem \ref{34} that $\G_k$ is finite. Now we assume 
$0\le i \le k-1$. 

By contradiction, there exists a sequence of measures $\{\mu_n\}_{n\in \NN}\subset \G_i$ and some invariant measure $\mu$ such that 
$$
\mu_n\xrightarrow{weak^*} \mu, \quad \textrm{as}~n\to +\infty.
$$ 
By Theorem \ref{34}, $\mu$ is a Gibbs $F_i$-state, which satisfies 
\begin{equation}\label{t}
\lambda^c_{i+1}(x)>0, \quad \mu \textrm{-a.e.}~ x\in M.
\end{equation}

%
%
On the other hand, since $\{\mu_n\}_{n\in \NN}\subset \G_i$, we know by definition that $\lambda^c_{i+1}(\mu_n)<0$ for every $n\in\NN$. The convergence $\mu_n\xrightarrow{weak^*} \mu$ yields
\begin{equation}\label{j}
\lambda^c_{i+1}(\mu)=\lim_{n\to +\infty}\lambda^c_{i+1}(\mu_n)\le 0.
\end{equation}
By Birkhoff's ergodic theorem, we also have
$$
\lambda^c_{i+1}(\mu)=\int \lambda^c_{i+1}(x) {\rm d}\mu.
$$
This together with (\ref{j}) yields that there exists a subset $A\subset M$ such that $\mu(A)>0$ and
$$
\lambda^c_{i+1}(x)<0, \quad \forall~ x\in A,
$$
which contradicts (\ref{t}) and we complete the proof.

\end{proof}


\begin{thebibliography}{ABV00}







%

\bibitem{AV15}
M. Andersson and C. Vasquez, On mostly expanding diffeomorphisms, {\it Ergod. Theory Dyn.
Syst.}, {\bf 38}(2018):2838-2859.




\bibitem{BDV05}
C. Bonatti, L. Diaz and M. Viana, Dynamics beyond uniform hyperbolicity, A global geometric and probabilistic perspective, Encyclopaedia
of Mathematical Sciences, 102. Mathematical Physics, 2005. III. Springer-Verlag, Berlin.

\bibitem{bp02}
L. Barreira, Y. Pesin, Lyapunov exponents and smooth ergodic theory, Univ. Lect. Ser. 23, American Mathematical Society, Providence RI, 2002.

%
%








\bibitem{Bow75}
R. Bowen, Equilibrium states and the ergodic theory of Anosov diffeomorphisms, Springer Lectures Notes in Math. (1975).

\bibitem{BoR75}
R. Bowen, D. Ruelle, The ergodic theory of Axiom A fows, {\it Invent. Math.}, {\bf29}(1975), 181-202.




\bibitem{bv}
C. Bonatti, M. Viana, SRB measures for partially hyperbolic systems whose central direction is mostly contracting, {\it Israel J. Math.},{\bf11} (2000), 157-193.






\bibitem{CMY22}
Y. Cao,  Z. Mi and D. Yang, On the abundance of Sinai-Ruelle-Bowen measures, {\it Commun. Math. Phys.}, {\bf 391}(2022), 1271-1306. 



\bibitem{cce15}
E. Catsigeras, M. Cerminara and H. Enrich, The Pesin entropy formula for $C^1$ diffeomorphisms with dominated splitting. {\it Ergod. Theory Dyn.
Syst.}, {\bf35} (2015), 737-761.





\bibitem{CYZ18}
S. Crovisier, D. Yang, J. Zhang, Empirical measures of partially hyperbolic attractors, {\it Commun.
Math. Phys.},{\bf 375} (2020), 725-764.

\bibitem{CSY15}
S. Crovisier, M. Sambarino and D. Yang, Partial hyperbolicity and homoclinic tangencies,  {\it J. Eur. Math. Soc.},{\bf 17} (2015), 1-49.





\bibitem{D12}
L. J. D\'iaz, T. Fisher, M. J. Pacifico and J. L. Vieitez, Entropy-expansiveness for partially
hyperbolic diffeomorphisms, {\it Discrete Contin. Dyn. Syst.}, {\bf 32} (2012),
4195-4207


%


\bibitem{HYY19}
Y. Hua, F. Yang and J. Yang, A new criterion of physical measures for partially hyperbolic diffeomorphisms,
{\it Trans. Am. Math. Soc.} {\bf 373(1)} (2020):385-417.

%
%

\bibitem{ly}
F. Ledrappier, L.-S. Young, The metric entropy of difeomorphisms Part I: Characterization of measures satisfying Pesin's entropy formula,
{\it Ann. Math.}, {\bf122}(1985), 509-539.

%
%
%


\bibitem{MCY17}
Z. Mi, Y. Cao and D. Yang, A note on partially hyperbolic systems with mostly expanding centers,
{\it Proc. Am. Math. Soc.}, {\bf 145(12)} (2017), 5299-5313.

%

\bibitem{pa}
J. Palis, A global view of Dynamics and a conjecture on the denseness of finitude
of attractors, {\it Ast$\acute{e}$risque}, {\bf261}(1999), 339-351.




\bibitem{ps82}
Y. Pesin, Y. Sinai, Gibbs measures for partially hyperbolic attractors, {\it Ergod. Theory Dyn.
Syst.},{\bf2} (1982), 417-438.



\bibitem{Rue76}
D. Ruelle, A measure associated with Axiom A attractors, {\it Amer. J. Math.}, {\bf98}(1976), 619-654.


\bibitem{Sin72}
Y. Sinai, Gibbs measures in ergodic theory, {\it Russ. Math.}, {\bf27(4)}(1972), 21-69.



\bibitem{cv}
C. V$\acute{a}$squez, Statistical stability for diffeomorphisms with dominated splitting, {\it Ergod. Theory Dyn.
Syst.},{\bf27} (2007): 253-283.






\bibitem{v2}
M.Viana., Dynamics: a probabilistic and geometric perspective, Documenta Mathematica ICM98,(1998) vol 1, 557-578.

%


\bibitem{YZ22}
D.Yang, Y. Zang, Volume Growth and Topological Entropy of Partially Hyperbolic Systems, {\it Israel J. Math.}, (2022), 1-23.

\bibitem{y02}
L.-S. Young, What are SRB measures, and which dynamical systems have them?, {\it J. Statist. Phys.}, {\bf108}(2002), 733-754.


%

\bibitem{wa82}
P. Walters, An introduction to ergodic theory, {\it Springer Verlag}, 1982.


\end{thebibliography}
\end{document}